\documentclass[11pt]{article}
\usepackage[margin=2.54cm]{geometry}
\usepackage{algorithmicx}

\usepackage[T1]{fontenc}
\usepackage[utf8]{inputenc}
\usepackage[english]{babel}
\usepackage{amsmath,amssymb, amsthm}
\usepackage[bookmarksopen,colorlinks,linkcolor=blue]{hyperref}
\usepackage{algorithm}
\usepackage{algpseudocode}
\usepackage{authblk}
\usepackage[shortlabels]{enumitem}
\usepackage{graphicx}
\usepackage{enumitem}
\usepackage{tikz}
\usetikzlibrary{decorations.text,calc,arrows.meta}
\usetikzlibrary{positioning} 
\usetikzlibrary{arrows,shapes}
\usetikzlibrary{chains,fit}
\usetikzlibrary{decorations.pathreplacing}

\usepackage{lmodern}
\usepackage{color}

\title{Energy Games over Totally Ordered Groups} 

\author{Alexander Kozachinskiy\thanks{kozlach@mail.ru. Steklov Mathematical Institute of Russian Academy of Sciences, 8 Gubkina St., Moscow 119991, Russia. This work was performed at the Steklov International Mathematical Center and supported by the Ministry of Science and Higher Education of the Russian Federation (agreement no. 075-15-2022-265).}
}

\newtheorem{theorem}{Theorem}

\newtheorem{proposition}[theorem]{Proposition}
\newtheorem{definition}{Definition}

\newtheorem{conjecture}{Conjecture}

\newcommand{\source}{\mathsf{source}}
\newcommand{\target}{\mathsf{target}}
\newcommand{\colo}{\mathsf{col}}

\newcommand{\val}{\mathbf{val}}
\newcommand{\sval}{\mathbf{\overline{val}}}

\newcommand{\per}{\mathsf{per}}

\theoremstyle{definition}

\begin{document}
\maketitle

\begin{abstract}
Kopczy\'{n}ski (ICALP 2006)
conjectured that prefix-independent half-positional winning conditions are closed under finite unions. We refute this conjecture over finite arenas. For that, we introduce a new class of prefix-independent bi-positional winning conditions called energy conditions over totally ordered groups. We give an example of two such conditions whose union is not half-positional. We also conjecture that every prefix-independent bi-positional winning condition coincides with some energy condition over a totally ordered group on periodic sequences.
\end{abstract}

\section{Introduction}

This paper is devoted to positional determinacy in turn-based infinite-duration games over \emph{finite} arenas. An arena is a finite directed graph whose edges are colored into elements of some finite set of colors $C$ and whose nodes are partitioned between two players called Alice and Bob. They play by traveling over the nodes of the arena. In each turn, one of the players chooses an edge from the current node, and the players move towards the endpoint of this edge. Whether it is an Alice's or a Bob's turn to choose depends on whether the current node is an Alice's node or a Bob's node. This continues for infinitely many turns. As a result, the players obtain an infinite word over $C$ (by concatenating colors of edges that appear in the play). A \emph{winning condition} $W$, which is a set of infinite words over $C$, defines the aims of the players. Alice wants to obtain an infinite word which belongs to $W$, while Bob wants it to be outside $W$.

A vast amount of literature in this area is devoted to \emph{positional strategies}. A strategy of Alice or Bob is positional if it never makes two different moves from the same node. Implementing such strategies is easy because we only have to specify one edge for each node of the corresponding player. This makes these strategies relevant for such areas as controller synthesis~\cite{bloem2018graph}, where an implementation of a controller can be seen as its strategy against an environment.

Correspondingly, of great interest are winning conditions for which positional strategies are always sufficient to play optimally (for one of the players or even for both of them). This area has the following terminology.
A winning condition $W$ is \emph{half-positional} if for every arena the following holds: either Alice has a positional winning strategy w.r.t.~$W$ or Bob has a winning strategy w.r.t.~$W$ (not necessarily positional). A winning condition $W$ is \emph{bi-positional} if additionally the same requirement as for Alice holds for Bob (or, in other words, if both $W$ and its complement are half-positional).

The most famous example of a bi-positional winning condition is the parity condition. It has an abundance of applications in logic, including decidability of logical theories~\cite{zielonka1998infinite} and modal $\mu$-calculus~\cite{gradel2002automata}.

Gimbert and Zielonka performed a general study of bi-positional winning conditions in a series of two papers~\cite{gimbert2004can,gimbert2005games}. In~\cite{gimbert2004can}, they gave a powerful sufficient condition for bi-positionality. It is suitable for almost all known bi-positional winning conditions. Moreover, in~\cite{gimbert2005games} they gave a sufficient and necessary condition for bi-positionality. Unfortunately, it is far more complex, and hence is less convenient for applications. Nevertheless, it has a corollary called \emph{1-to-2-player lifting}, which is of great interest in practice. It states that as long as a winning condition  is half-positional for Alice in arenas without Bob and half-positional for Bob in arenas without Alice, it is bi-positional in all arenas.

 At the same time, several strong sufficient conditions for half-positionality were obtained in the literature~\cite{kopczynski2006half,bianco2011exploring}, but none of them is necessary. In general, half-positionality is far less understood than bi-positionality. For instance, Kopczy\'{n}ski~\cite{kopczynski2006half} made the following simple-looking conjecture, which was open for more than 15 years.

 \begin{conjecture}[Kopczy\'{n}ski, \cite{kopczynski2006half}]
\label{conj}
Prefix-independent half-positional winning conditions are closed under finite unions. (``Prefix-independent'' means  closed under removing and adding finite prefixes.)
\end{conjecture}
He refuted this conjecture for uncountable unions. However, the  assortment of half-positional winning conditions at the time was not wide enough to refute it even for countable unions.

 In this paper, we refute  Conjecture \ref{conj}. Moreover,  we show that there are two \emph{bi-positional} prefix-independent winning conditions whose union is not \emph{half-positional}.
To this end, we introduce a new class of bi-positional winning conditions called \emph{energy conditions over totally ordered groups}, or ETOG conditions for short. They are defined as follows (see more details in Section \ref{sec:def}). We consider elements of some totally ordered group (we stress that it should be bi-ordered) as colors of edges. Given an infinite sequence of these elements, we arrange them into a formal series. Alice wants the sequence of its partial sums to have an infinite decreasing subsequence. This generalizes canonical energy\footnote{In case of $\mathbb{Z}$, a color of an edge is interpreted as an amount of energy needed to pass through this edge (and negative edges are edges where one can recharge). In this sense, Bob wins if there is a finite amount of initial energy allowing him to never run out of it.} conditions~\cite{chakrabarti2003resource} that are defined over $\mathbb{Z}$ with the standard ordering.

In Section \ref{sec:pos}, we establish bi-positionality of energy conditions over totally ordered group using a sufficient condition of Gimbert and Zielonka. Next, we refute Conjecture \ref{conj} in Section \ref{sec:refute}. A key factor allowing us to do this is that free groups can be totally ordered. We construct two energy conditions over a free group with 2 generators whose union is not half-positional. We also observe in Section \ref{sec:refute} that energy conditions over free groups are non-permuting, and that they can be used to refute 1-to-2-player lifting for half-positionality.

We believe that the class of energy conditions over totally ordered groups is interesting on its own. Namely, we find this class suitable for the following conjecture.
\begin{conjecture}
\label{our_conj}
Every bi-positional prefix-independent winning condition coincides on periodic sequences with some energy condition over a totally ordered group.
\end{conjecture}
We cannot expect it to hold for all sequences, but periodic once are sufficient, say, for algorithmic applications. 
If our conjecture is true, it gives an explicit description of the class of bi-positional prefix-independent winning condition. This would in line with an explicit description of the class of \emph{continuous} bi-positional payoffs from~\cite{kozachinskiy:LIPIcs.CONCUR.2021.10}.
We discuss our conjecture in more detail in Section \ref{sec:conj}, where we reduce it to a problem about free groups.

\medskip

\textbf{Open problems and related works.} Conjecture \ref{conj} is still open for \emph{infinite arenas}, despite a recent breakthrough by  Ohlmann~\cite{pierrephd}, who obtained a characterization of the infinite-arena  half-positionality in terms of well-founded universal graphs. Unfortunately, ETOG conditions are generally not half-positional over infinite arenas.

Another interesting open question is whether
Conjecture \ref{conj} holds in restriction to $\omega$-regular conditions.  In a recent preprint~\cite{https://doi.org/10.48550/arxiv.2205.01365}, Bouyer et al.~(among other results) give a positive answer to this question for  $\omega$-regular conditions, recognizable by deterministic B\"{u}chi automata (DBA). In fact, they simply show that every prefix-independent DBA-recognizable  $\omega$-regular condition can be given as a set of sequences, having infinitely many occurrences of some fixed subset of colors. Such conditions are trivially closed under finite unions. General prefix-independent $\omega$-regular conditions might be arranged in a far more complex way (even though there is an algorithm due to Kopczy\'{n}ski~\cite{kopczynski2007omega}, deciding half-positionality for them).

\section{Preliminaries}

If $C$ is a set, we denote by $C^*$ (resp., by $C^\omega$) the set of all finite (resp., infinite) words over $C$. For $x\in C^*$, by $|x|$ we denote the length of $x$. Additionally, by $C^+$ we denote the set of all finite non-empty words over $C$. If $x\in C^+$, then by $x^\omega$ we denote an infinite word obtained by repeating $x$ infinitely many times. The free group over $C$ is denoted by $F_C$.

\medskip

An arena $\mathcal{A}$ over a non-empty finite set (of colors) $C$ is a tuple $\langle V_A, V_B, E\rangle$, where $V_A$ and $V_B$ are disjoint finite sets and $E\subseteq (V_A\cup V_E) \times C \times (V_A\cup V_B)$ is such that for every $s\in V_A\cup V_B$ there exist $c\in C$ and $t\in V_A\cup V_B$ for which $(s, c, t)\in E$. Elements of $V_A$ are called Alice's nodes, and elements of $V_B$ are called Bob's nodes. Elements of $E$ are called edges of $\mathcal{A}$. An edge $e = (s, c, t)\in E$ is represented as a $c$-colored arrow from $s$ to $t$. We use the notation $\source((s, c, t)) = s, \colo((s, c, t)) = c$ and $\target((s, c, t)) = t$. Our definition guaranties that every node $v\in V_A\cup V_B$ has an out-going edge, that is, an edge $e$ such that $\source(e) = v$.

An infinite-duration game over $\mathcal{A}$ from a node $s\in V_A\cup V_B$ is played as follows. At the beginning, one of the players chooses an edge $e_1\in E$ with $\source(e_1) = s$. Namely, if $s\in V_A$, then Alice chooses $e_1$, and if $s\in V_B$, then Bob chooses $e_1$. More generally, in the first $n$ turns players choose $n$ edges $e_1, e_2, \ldots, e_n\in E$, one edge per turn. These edges always form a \emph{path} in $\mathcal{A}$, that is, we have $\target(e_1) = \source(e_2), \ldots, \target(e_{n-1}) = \source(e_n)$. Then the $(n + 1)$st turn is played as follows. Players consider the endpoint node of the current path, which is $\target(e_n)$. One of the players chooses an edge $e_{n + 1}$ with $\source(e_{n + 1}) = \target(e_n)$. Namely, if $\target(e_n) \in V_A$, then Alice chooses $e_{n + 1}$, and if $\target(e_n)\in V_B$, then Bob chooses $e_{n + 1}$. After infinitely many turns, players get an infinite sequence of edges $p = (e_1, e_2, e_3, \ldots)$ called a play (it forms an infinite path in $\mathcal{A}$).

A winning condition over a set of colors $C$ is a subset $W\subseteq C^\omega$. A strategy of Alice is winning from $s\in V_A\cup V_B$ w.r.t.~$W$ if any play $p = (e_1, e_2, e_3, \ldots)$ with this strategy in the infinite-duration game over $\mathcal{A}$ from $s$ is such that its sequence of colors $\colo(e_1)\colo(e_2)\colo(e_3)\ldots$ belongs to $W$. Similarly, a strategy of Bob is winning from $s\in V_A\cup V_B$ w.r.t.~$W$ if any play $p = (e_1, e_2, e_3, \ldots)$ with this strategy in the infinite-duration game over $\mathcal{A}$ from $s$ is such that $\colo(e_1)\colo(e_2)\colo(e_3)\ldots\notin W$.

A positional strategy of Alice is a function $\sigma\colon V_A\to E$ such that $\source(\sigma(u)) = u$ for any $u\in V_A$. It is interpreted as follows: for any $u\in V_A$, whenever Alice has to choose an edge from $u$, she chooses $\sigma(u)$. Similarly, a positional strategy of Bob is a function  $\tau\colon V_B\to E$ such that $\source(\tau(u)) = u$ for any $u\in V_B$. It is interpreted analogously.

 A winning condition $W\subseteq C^\omega$ is \emph{half-positional} if for every finite arena $\mathcal{A}$ over $C$ there exists a positional strategy $\sigma$ of Alice such that for every node $s$ of $\mathcal{A}$ the following holds: if $\sigma$ is not winning w.r.t.~$W$ from $s$, then Bob has a winning strategy w.r.t.~$W$ from $s$. A winning condition $W$ is \emph{bi-positional} if both $W$ and its complement $C^\omega\setminus W$ are half-positional.

A winning condition $W\subseteq C^\omega$ is \emph{prefix-independent} if for all $x\in C^*$ and $\alpha\in C^\omega$ we have $\alpha\in W\iff x\alpha \in W$.

We state the following sufficient condition for bi-positionality due to Gimbert and Zielonka.

\begin{definition} 
\label{def:fairly_mixing}
Let $W\subseteq C^\omega$ be a winning condition over a finite set of colors $C$. We call $W$ \textbf{fairly mixing} if the following 3 conditions hold:
\begin{enumerate}[label=\Alph*)]
\item For every $x\in C^*$ and $\alpha, \beta\in C^\omega$ we have that 
\[(x\alpha\notin W \land x\beta\in W) \implies (\alpha\notin W \land \beta\in W).\]
\item For every $S\in\{W, C^\omega\setminus W\}$, for every $x\in C^+$ and for every $\alpha\in C^\omega$ we have that
\[(x^\omega\in S, \alpha\in S)\implies (x\alpha\in S).\]
\item For every $S\in\{W, C^\omega\setminus W\}$ and for every infinite sequence $x_1, x_2, x_3, \ldots\in C^+$ it holds that:
\begin{align*}
\bigg[( x_1 x_3 x_5\ldots\in S\big) \land (x_2 x_4 x_6\ldots\in S) \land(\forall n \ge 1\,\, x_n^\omega\in S) \bigg] \implies x_1 x_2 x_3\ldots\in S.
\end{align*}
\end{enumerate}
\end{definition}

\begin{theorem}[\cite{gimbert2004can}]
\label{thm:gztheorem}
 Any fairly mixing winning condition is bi-positional.
\end{theorem}

\section{Definition of Energy Games over Totally Ordered Groups}
\label{sec:def}

Recall that a totally ordered group is a triple $(G, +, \le)$, where $(G, +)$ is a group and $\le$ is a total order on $G$ such that
\[a\le b \implies x + a + y \le x + b + y \qquad \mbox{for all } a, b, x, y\in G.\]

Consider any finite set $C$ of colors and any totally ordered group $(G, +,\le)$. By a \emph{valuation of colors} over $(G, +, \le)$ we mean any function $\val\colon C\to G$. It can be extended to a homomorphism $\val\colon C^*\to G$ by setting
\[\val(\mbox{empty word}) = 0, \qquad \val(c_1 c_2\ldots c_n) = \val(c_1) + \val(c_2) + \ldots + \val(c_n).\]
Additionally, for every infinite sequence of colors $c_1 c_2 c_3\ldots\in C^\omega$, we denote by $\sval(c_1 c_2 c_3\ldots)$ the sequences of valuations of its finite prefixes: 
\[\sval(c_1 c_2 c_3\ldots) = \{\val(c_1\ldots c_n)\}_{n = 1}^\infty.\]
In other words,  $\sval(c_1 c_2 c_3\ldots)$ is the sequence of  partial sums of the formal series $\sum_{n = 1}^\infty \val(c_n)$.

\emph{An energy condition over  $(G, +, \le)$}, defined by a valuation of colors $\val\colon C\to G$, is the set $W\subseteq C^\omega$ of all $\alpha\in C^\omega$ such that $\sval(\alpha)$
has an infinite decreasing subsequence. It is immediate that any energy condition over a totally ordered group is prefix-independent.

As an illustration, we show that parity conditions fall into this definition. The parity condition over $d$ priorities is a winning condition $W_{par}^d\subseteq\{1,2,\ldots,d\}^\omega$, 
\[W_{par}^d = \{ c_1 c_2 c_3\ldots\in \{1, 2, \ldots, d\}^\omega\mid \limsup_{n\to\infty} c_i \mbox{ is odd}\}.\]
Observe that $W_{par}^d$ is an energy condition over $\mathbb{Z}^d$ with the lexicographic ordering, defined by the following valuation:
\begin{align*}
    \val(d) &= ((-1)^d, 0, \ldots 0) \\
    \val(d - 1) &= (0, (-1)^{d - 1}, \ldots 0)\\
    &\vdots\\
     \val(1) &= (0, 0, \ldots, -1).
\end{align*}

As far as we know, the most general class of bi-positional prefix-independent winning conditions that were previously considered are \emph{priority mean payoff conditions}~\cite{gimbert2006deterministic}. They can also be defined as energy conditions over $\mathbb{Z}^d$. Moreover, to define them, it is sufficient to consider only valuations that map each color to a vector with at most 1 non-zero coordinate, as in the case of parity conditions.

\section{Bi-positionality of Energy Conditions over Totally Ordered Groups}
\label{sec:pos}
In this section, we establish

\begin{theorem}
\label{thm:pos}
Every ETOG condition  is bi-positional.
\end{theorem}

We derive it from the following technical result (which will also be useful in Section \ref{sec:conj}). If $C$ is a non-empty finite set and $W\subseteq C^\omega$, define $\per(W) = \{x\in C^+\mid x^\omega \in W\}$ to be the set of periods of periodic words from $W$.

\begin{proposition}
\label{prop:technical}
Let $C$ be a non-empty finite set. Consider any set $P\subseteq C^+$ such that both $P$ and $C^+\setminus P$ are closed under concatenations and cyclic shifts. Define a winning condition $W_P\subseteq C^\omega$ as follows:
\[W_P = \{xy_1 y_2 y_3\ldots \mid x\in C^*, y_1, y_2, y_3,\ldots\in P\}.\]
Then $W_P$ is a prefix-independent fairly mixing winning condition with $P = \per(W_P)$.
\end{proposition}

Let us start with a derivation of Theorem \ref{thm:pos}.
\begin{proof}[Proof of Theorem \ref{thm:pos} (modulo Proposition \ref{prop:technical})]
Assume that $W\subseteq C^\omega$ is an energy condition over a totally ordered group $(G, +, \le)$, defined by a valuation of colors $\val\colon C \to G$. Set $P = \{y\in C^+ \mid \val(y) < 0\}$. We claim that $W = W_P$. Indeed, $W$ consists of all $\alpha = c_1 c_2 c_3\ldots\in C^\omega$ such that 
\[\sval(\alpha) = (\val(c_1), \val(c_1 c_2), \val(c_1 c_2 c_3), \ldots)\]
has an infinite decreasing subsequence. Consider any $i < j$. Observe that the $j$th element of $\sval(\alpha)$ is smaller than the $i$th element of $\sval(\alpha)$ if and only if 
\[-\val(c_1\ldots c_i) + \val(c_1 \ldots c_j)  = \val(c_{i + 1}\ldots c_j) < 0.\]
In other words, $\sval(\alpha)$ has an infinite decreasing subsequence if and only if $\alpha = c_1 c_2 c_3\ldots$ can be represented, except for some finite prefix, as a as a sequence of words with negative valuations. This means that $W = W_P$.

We now show that both $P$ and $C^+\setminus P$ are closed under concatenations and cyclic shifts. By Proposition \ref{prop:technical}, this would imply that $W = W_P$ is fairly mixing. In turn, by Theorem \ref{thm:gztheorem}, this implies that $W$ is bi-positional.

Consider any two words $x, y\in C^+$. Obviously:
\begin{align*}
    \val(x) < 0, \val(y) < 0 &\implies \val(xy) = \val(x) + \val(y) < 0,\\
    \val(x) \ge 0, \val(y) \ge 0 &\implies \val(xy) = \val(x) + \val(y) \ge 0.
\end{align*}
This demonstrates that both $P$ and $C^+\setminus P$ are closed under concatenations. Now, we claim that $\val(c_1 c_2\ldots c_n) < 0 \iff \val(c_2\ldots c_n c_1) < 0$ for any word $c_1c_2\ldots c_n\in C^+$ (this implies that both $P$ and $C^+\setminus P$ are closed under cyclic shifts). Indeed,
\begin{align*}
     &\val(c_1) + \val(c_2) +  \ldots + \val(c_n) < 0 \\
    &\iff -\val(c_1) + (\val(c_1) + \val(c_2) +  \ldots + \val(c_n)) + \val(c_1) < -\val(c_1) + 0 + \val(c_1)\\
    &\iff \val(c_2) + \ldots + \val(c_n) + \val(c_1) < 0.
\end{align*}
\end{proof}

\begin{proof}[Proof of Proposition \ref{prop:technical}]
Prefix-independence of $W_P$ is immediate. We now show that $P = \per(W_P)$. We have $z^\omega\in W_P$ for any $z\in P$ by definition. Hence, $P\subseteq \per(W_P)$. Now, take any $z\in \per(W_P)$. We show that $z\in P$. By definition of $\per(W_P)$, we have $z^\omega = xy_1 y_2 y_3\ldots$ for some $x\in C^*$ and $y_1,y_2,y_3\ldots \in P$. There exist $i < j$ such that $|xy_1\ldots y_i|$ and $|xy_1\ldots y_j|$ are equal modulo $|z|$. This means that $y_{j+1} \ldots y_j$ must be a multiple of some cyclic shift of $z$. We have that $y_{j+1} \ldots y_j\in P$ because $P$ is closed under concatenations. This means that this cyclic shift of $z$ also belongs to $P$. Indeed, otherwise we could write  $y_{j+1} \ldots y_j$ as a multiple of some word from $C^+\setminus P$, and this is impossible because $C^+\setminus P$ is closed under concatenations. Since $P$ is closed under cyclic shifts, we obtain $z\in P$.

Finally, we show that $W_P$ is fairly mixing. Since $W_P$ is prefix-independent, we should care only about the third item of Definition \ref{def:fairly_mixing}. That is, we only have to show the following two claims:  

\begin{align}
\label{eq:simple}
&\bigg[( x_1 x_3 x_5\ldots\in W_P) \land (x_2 x_4 x_6\ldots\in W_P) \land(\forall n\ge1\,\, x_n^\omega\in W_P) \bigg] \implies x_1 x_2 x_3\ldots\in W_P,\\
\label{eq:hard}
&\bigg[( x_1 x_3 x_5\ldots\in \overline{W_P}) \land (x_2 x_4 x_6\ldots\in \overline{W_P}) \land(\forall n\ge 1\,\, x_n^\omega\in \overline{W_P}) \bigg] \implies x_1 x_2 x_3\ldots\in \overline{W_P},
\end{align}
for every infinite sequence of words $x_1, x_2, x_3,\ldots \in C^+$. Here, for brevity, by $\overline{W_P}$ we denote $C^\omega\setminus W_P$.

We first show \eqref{eq:simple}. If $x_n^\omega\in W_P$ for every $n$, then $x_n\in \per(W_P) = P$ for every $n$, and hence $x_1x_2x_3\ldots\in W_P$ by definition.

A proof of \eqref{eq:hard} is more elaborate. Assume for contradiction that $x_1 x_2 x_3\ldots \in W_P$. Then we can write $x_1 x_2 x_3\ldots = x y_1 y_2 y_3\ldots $ for some $x\in C^*$ and $y_1, y_2, y_3,\ldots\in P$. One can represent the equality as a sequence  of ``cuts'' inside $x_1 x_2 x_3\ldots$, as on the following picture:

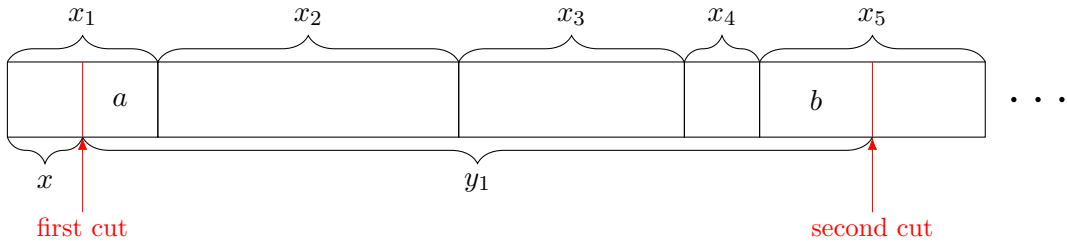
\begin{center}
\begin{tikzpicture}
\draw[draw=black] (0,0) rectangle ++(2,1);

\draw[draw=black] (2,0) rectangle ++(4,1);

\draw[draw=black] (6,0) rectangle ++(3,1);

\draw[draw=black] (9,0) rectangle ++(1,1);

\draw[draw=black] (10,0) rectangle ++(3,1);

\draw[draw=red] (1,0) -- (1,1);

\draw[draw=red] (11.5,0) -- (11.5,1);

\draw [decorate,decoration={brace,amplitude=10pt},xshift=0pt,yshift=0pt]
(0,1) -- (2,1) node [black,midway,yshift=0.6cm] {$x_1$};

\draw [decorate,decoration={brace,amplitude=10pt},xshift=0pt,yshift=0pt]
(2,1) -- (6,1) node [black,midway,yshift=0.6cm] {$x_2$};

\draw [decorate,decoration={brace,amplitude=10pt},xshift=0pt,yshift=0pt]
(6,1) -- (9,1) node [black,midway,yshift=0.6cm] {$x_3$};

\draw [decorate,decoration={brace,amplitude=10pt},xshift=0pt,yshift=0pt]
(9,1) -- (10,1) node [black,midway,yshift=0.6cm] {$x_4$};

\draw [decorate,decoration={brace,amplitude=10pt},xshift=0pt,yshift=0pt]
(10,1) -- (13,1) node [black,midway,yshift=0.6cm] {$x_5$};

	\draw [-Latex, color=red](1,-1) -- (1,0);

	\draw [-Latex, color=red](11.5,-1) -- (11.5,0);

\draw [decorate,decoration={brace,amplitude=10pt, mirror},xshift=0pt,yshift=0pt]
(0,0) -- (1,0) node [black,midway,yshift=-0.6cm] {$x$};

\draw [decorate,decoration={brace,amplitude=10pt, mirror},xshift=0pt,yshift=0pt]
(1,0) -- (11.5,0) node [black,midway,yshift=-0.6cm] {$y_1$};

\node[draw=none] (a) at (1.5, 0.5) {$a$};

\node[draw=none] (b) at (10.75, 0.5) {$b$};

\node[draw=none] (a) at (1, -1.2) {\small \textcolor{red}{first cut}};
\node[draw=none] (a) at (11.5, -1.2) {\small \textcolor{red}{second cut}};

\node[draw=none] (a) at (13.7, 0.5) {\huge $\ldots$};
\end{tikzpicture}
\end{center}

Either there are infinitely many cuts inside $x_n$ with odd indices, or there are infinitely many cuts inside $x_n$ with even indices. Without loss of generality, we may assume that we only have cuts inside $x_n$ with odd indices, and at most one for each $n$. Indeed, if necessary, we can join several successive $y_i$'s into one word (this is legal because $P$ is closed under concatenations).

We can now write each $y_i$ as $y_i = ax_{2k} x_{2k + 1}\ldots x_{2m}b$ for some $a,b\in C^*$ and $1\le k\le m$. Now, let $y_i^\prime = a x_{2k + 1} x_{2k + 3} \ldots x_{2m - 1} b$ be a word which can be obtained from $y_i$ by removing $x_n$ with even indices. Additionally, we let $x^\prime\in C^*$ be a word 
which can be obtained from $x$ in the same way.
 Since each $x_n$ with an even index lies entirely in some $y_i$ or in $x$, we have that $x_1x_3x_5\ldots = x^\prime y_1^\prime y_2^\prime y_3^\prime\ldots$, as the following picture illustrates:
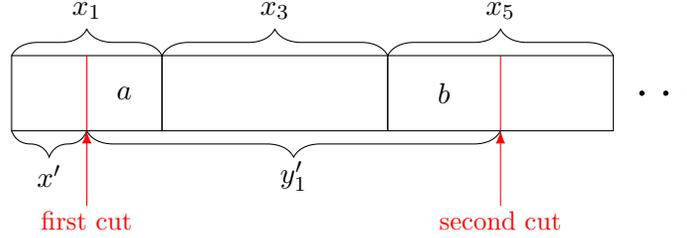
\begin{center}
\begin{tikzpicture}
\draw[draw=black] (0,0) rectangle ++(2,1);

\draw[draw=black] (2,0) rectangle ++(3,1);

\draw[draw=black] (5,0) rectangle ++(3,1);

\draw[draw=red] (1,0) -- (1,1);

\draw[draw=red] (6.5,0) -- (6.5,1);

\draw [decorate,decoration={brace,amplitude=10pt},xshift=0pt,yshift=0pt]
(0,1) -- (2,1) node [black,midway,yshift=0.6cm] {$x_1$};

\draw [decorate,decoration={brace,amplitude=10pt},xshift=0pt,yshift=0pt]
(2,1) -- (5,1) node [black,midway,yshift=0.6cm] {$x_3$};

\draw [decorate,decoration={brace,amplitude=10pt},xshift=0pt,yshift=0pt]
(5,1) -- (8,1) node [black,midway,yshift=0.6cm] {$x_5$};

\draw [decorate,decoration={brace,amplitude=10pt, mirror},xshift=0pt,yshift=0pt]
(0,0) -- (1,0) node [black,midway,yshift=-0.6cm] {$x^\prime$};

\draw [decorate,decoration={brace,amplitude=10pt, mirror},xshift=0pt,yshift=0pt]
(1,0) -- (6.5,0) node [black,midway,yshift=-0.6cm] {$y_1^\prime$};

	\draw [-Latex, color=red](1,-1) -- (1,0);

	\draw [-Latex, color=red](6.5,-1) -- (6.5,0);

\node[draw=none] (a) at (1.5, 0.5) {$a$};

\node[draw=none] (b) at (5.75, 0.5) {$b$};

\node[draw=none] (a) at (1, -1.2) {\small \textcolor{red}{first cut}};
\node[draw=none] (a) at (6.5, -1.2) {\small \textcolor{red}{second cut}};

\node[draw=none] (a) at (8.7, 0.5) {\huge $\ldots$};
\end{tikzpicture}
\end{center}
We will show that $y_i^\prime\in P$ for every $P$. This would contradict a fact that $x_1 x_3 x_5\ldots \in \overline{W_P}$.

First, observe that $x_n\notin P$ for every $n$. Indeed, we are given that $x_n^\omega\in \overline{W_P}$ for every $n$. Hence, $x_n \notin \per(W_P) = P$, as required.

Assume for contradiction that $y_i^\prime = ax_{2k + 1} x_{2k + 3}\ldots x_{2m - 1}b\notin P$.  Using a fact that $C^+\setminus P$ is closed under concatenations and cyclic shifts, we obtain:

\begin{align*}
y_i^\prime = &ax_{2k + 1} x_{2k + 3}\ldots x_{2m - 1} b\notin  P \\
\implies &x_{2k + 1} x_{2k + 3}\ldots x_{2m - 1} ba\notin P \\
\implies &x_{2k}x_{2k + 1} x_{2k + 3}\ldots x_{2m - 1} ba\notin  P && \mbox{because } x_{2k}\notin  P\\
\implies & x_{2k + 3}\ldots x_{2m - 1} ba x_{2k} x_{2k + 1}\notin  P \\
\ \implies &x_{2k + 2}x_{2k + 3}\ldots x_{2m - 1} ba x_{2k} x_{2k + 1}\notin  P && \mbox{because } x_{2k + 2}\notin  P\\
&\vdots\\
\implies &x_{2m} ba x_{2k} x_{2k + 1}\ldots x_{2m - 1} \notin P && \mbox{because } x_{2m}\notin  P\\
\implies &y_i = a x_{2k} x_{2k + 1}\ldots x_{2m} b\notin P,
\end{align*}
contradiction.
\end{proof}

\section{Refuting Conjecture \ref{conj}}
\label{sec:refute}

Consider the free group $F_{\{a, b\}}$ with 2 generators $a, b$. As was proved by Shimbireva~\cite{shimbireva1947theory}, see also~\cite[Page 18]{deroin2014groups}, free groups can be totally ordered. We take an arbitrary total ordering $\le$ of $F_{\{a, b\}}$. We also consider its inverse $\le^{-1}$, which is also a total ordering of $F_{\{a, b\}}$. Define a set of colors $C = \{a, a^{-1}, b, b^{-1}, \varepsilon\}$. Here $a^{-1}, b^{-1}$ are inverses of $a, b$ in $F_{\{a, b\}}$, and $\varepsilon$ is the identity element of $F_{\{a,b\}}$.

Let $W_1\subseteq C^\omega$ be an energy condition over $(F_{\{a, b\}}, \le)$, defined by a (suggestive) valuation of colors which interprets elements of $C$ as corresponding elements of $F_{\{a,b\}}$. Similarly, we let $W_2\subseteq C^\omega$ be an energy condition over $(F_{\{a, b\}}, \le^{-1})$, defined by the same valuation. The only difference between $W_1$ and $W_2$ is that they are defined w.r.t.~different total orderings of $F_{\{a, b\}}$ (one ordering is the inverse of the other one).

We show that the union $W_1\cup W_2$ is not half-positional. It consists of all $\alpha \in C^\omega$ such that $\sval(\alpha)$  contains either an infinite decreasing subsequence w.r.t.~$\le$ or an infinite decreasing subsequence w.r.t.~$\le^{-1}$. In other words, it consists of  all $\alpha \in C^\omega$ such that $\sval(\alpha)$  contains either an infinite decreasing subsequence or an infinite \emph{increasing} subsequence w.r.t.~$\le$.

We show that $W_1\cup W_2$ is not half-positional in the following arena.

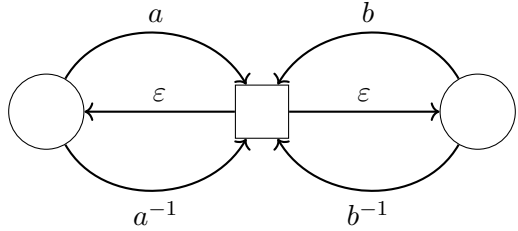
\begin{center}
\begin{tikzpicture}
\node[draw,  regular polygon, regular polygon sides=4, minimum size=1cm] (1) {};

\node[draw,  circle, right=2cm of 1,minimum size=1cm] (2) {};

\node[draw,  circle, left=2cm of 1, minimum size=1cm] (3) {};

\draw[thick,->] (1) -- (2) node[midway, above] {$\varepsilon$};

\draw[thick,->] (1) -- (3) node[midway, above] {$\varepsilon$};

\path[->]
 (2) edge [thick, in=60,out=120,out distance=1cm,in distance=1cm] node[midway, above] {$b$} (1);

\path[->]
 (2) edge [thick, in=-60,out=-120,out distance=1cm,in distance=1cm] node[midway, below] {$b^{-1}$} (1);

\path[->]
 (3) edge [thick, in=120,out=60,out distance=1cm,in distance=1cm] node[midway, above] {$a$} (1);

\path[->]
 (3) edge [thick, in=-120,out=-60,out distance=1cm,in distance=1cm] node[midway, below] {$a^{-1}$} (1);
\end{tikzpicture}
\end{center}

Here, Alice controls the square and Bob controls the two circles.
Assume that the game starts in the square. We show that Alice has a winning strategy w.r.t.~$W_1\cup W_2$, but not a positional one.

Alice has two positional strategies in this arena: always go to the left and always go to the right. Consider, for example, the first one. Bob has the following counter-strategy which wins against it: alternate the $a$-edge with the $a^{-1}$-edge. We get the following sequence of colors in the play of these two strategies:
\[\varepsilon a \varepsilon a^{-1} \varepsilon a \varepsilon a^{-1}\ldots\]
This sequence does not belong to $W_1\cup W_2$ because
\[\sval(\varepsilon a \varepsilon a^{-1} \varepsilon a \varepsilon a^{-1}\ldots) = \varepsilon, a, a, \varepsilon, \varepsilon, a, a, \varepsilon,\ldots \]
There are only two distinct elements of $F_{\{a,b\}}$ occurring in $\sval(\varepsilon a \varepsilon a^{-1} \varepsilon a \varepsilon a^{-1}\ldots)$. Hence, it neither has an infinite decreasing subsequence nor an infinite increasing subsequence. By the same argument, the second positional strategy of Alice (always go to the right) is not winning w.r.t.~$W_1\cup W_2$ either.

On the other hand, Alice has the following winning strategy: alternate the edge to the left circle with the edge to the right circle. Consider any play with this strategy. Its sequence of colors looks as follows:
\[\varepsilon a^{\pm 1} \varepsilon b^{\pm1} \varepsilon a^{\pm 1} \varepsilon b^{\pm1}\ldots\]
We show that this sequence belongs to $W_1\cup W_2$.
A restriction of $\sval(\varepsilon a^{\pm 1} \varepsilon b^{\pm1} \varepsilon a^{\pm 1} \varepsilon b^{\pm1})$ to elements with even indices looks like this:
\begin{equation}
\label{eq:sec}
a^{\pm 1}, \,\, a^{\pm 1}b^{\pm 1}, \,\,  a^{\pm 1}b^{\pm 1} a^{\pm 1}, \,\, a^{\pm 1}b^{\pm 1} a^{\pm 1}b^{\pm 1} \ldots
\end{equation}

All elements of \eqref{eq:sec} are distinct. Hence, by the Infinite Ramsey Theorem, it either has an infinite decreasing subsequence or an infinite increasing subsequence w.r.t.~$\le$. Indeed, consider an infinite complete graph over $\{1, 2, 3, \ldots\}$, whose edges are colored into green and red as follows. Pick any $i, j\in\{1, 2, 3,\ldots\}$, $i < j$. If the $i$th element of \eqref{eq:sec} is bigger than the $j$th element of \eqref{eq:sec}, then color the edge between $i$ and $j$ into green. Otherwise, color this edge into red (in this case,  the $i$th element of \eqref{eq:sec} is smaller than the $j$th element of \eqref{eq:sec}). Our graph has an infinite induced subgraph in which all edges are of the same color. If they are all green (resp., red), then this subgraph defines an infinite decreasing (resp., increasing) subsequence of \eqref{eq:sec}.

\bigskip

\textbf{Additional remarks.} Energy conditions over free groups are interesting because they are \emph{non-permuting} (if there is more than one generator). A prefix-independent winning condition is permuting if it is closed under permuting periods of periodic sequences. All previously known prefix-independent bi-positional winning condition were permuting. This is because they can be seen as energy conditions over abelian groups (on periodic sequences). In a talk of Colcombet and Niwi\'{n}ski~\cite{slides} it was asked whether there exists a non-permuting bi-positional prefix-independent winning condition. The answer is ``yes''. For example, take $W_1$ as above in this section. Without loss of generality, we may assume that $aba^{-1} b^{-1}$ is negative w.r.t.~$\le$ (otherwise we can consider its inverse). Then $(aba^{-1}b^{-1})^\omega\in W_1$, but $(aa^{-1} b b^{-1})^\omega\notin W_1$.

Additionally, the winning condition $W_1\cup W_2$ is interesting because it refutes  1-to-2-player lifting for half-positionality. Namely, it is easy to see that $W_1\cup W_2$ is positional for Alice in all arenas, where there are no nodes of Bob. This is because she can win in such arenas if and only if there is a reachable non-zero simple cycle. But as we have shown, $W_1\cup W_2$ is not positional for Alice in the presence of Bob. Previously, there were examples that refute 1-to-2-player lifting for half-positionality in stochastic games~\cite{gimbert2014two}.
\section{Discussing Conjecture \ref{our_conj}}
\label{sec:conj}

First, it is useful to understand how prefix-independent bi-positional winning condition are arranged on periodic sequences.
 Luckily, Proposition \ref{prop:technical} gives an answer.

\begin{proposition}
\label{prop:periods}
 Let $C$ be a finite non-empty set. Then for any $P\subseteq C^+$ the following two conditions are equivalent:
\begin{enumerate}[label=\Alph*)]
\item $P = \per(W)$ for some prefix-independent bi-positional winning condition $W\subseteq C^\omega$;
\item $P$ and $C^+\setminus P$ are closed under concatenations and cyclic shifts;
\end{enumerate}

\end{proposition}
\begin{proof}
The fact that the second item implies the first item follows from Proposition \ref{prop:technical}. Indeed, if $P$ and $C^+\setminus P$  are closed under concatenations and cyclic shifts, then $P = \per(W_P)$ for a prefix-independent fairly mixing winning condition $W_P$, which is bi-positional by Theorem \ref{thm:gztheorem}. We now show that the first item implies the second item. The fact that $P$ and $C^+\setminus P$ are closed under cyclic shifts is a consequence of the prefix-independence of $W$:
\begin{align*}
    c_1 c_2 \ldots c_n \in P &\iff (c_1 c_2 \ldots c_n)^\omega \in W \iff c_n (c_1 c_2 \ldots c_n)^\omega = (c_n c_1\ldots c_{n - 1})^\omega \in W\\
    &\iff c_n c_1\ldots c_{n - 1}\in P.
\end{align*}
We now show that $P$ is closed under concatenations (there is a similar argument for $C^+\setminus P$). Take any $x,y\in P$. Consider the following arena.

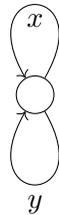
\begin{center}
\begin{tikzpicture}

\node[draw,  circle, right=2cm of 1,minimum size=0.5cm] (A) {};

\path[->] (A) edge [anchor=center,loop below, looseness=15, out=60, in=120] node {$x$} (A);

\path[->] (A) edge [anchor=center,loop below, looseness=15, out=-60, in=-120] node {$y$} (A);
\end{tikzpicture}
\end{center}

It has a central circle node which lies on two simple cycles, one of which is colored by $x$ and the other one by $y$. All nodes are controlled by Bob. Since, $x, y\in P$, we have that $x^\omega,y^\omega\in W$. Hence, Bob does not have a positional winning strategy w.r.t.~$W$ from the central circle. Since $W$ is bi-positional, Bob has no winning strategy from the central circle w.r.t.~$W$. Now, assume that Bob alternates the $x$-cycle with the $y$-cycle. He obtains $(xy)^\omega$ as a sequence of colors. Since this strategy is not winning, we have $xy\in P$.
\end{proof}

In turn, periods of periodic sequences of ETOG conditions are arranged as follows.

\begin{proposition}
Let $C$ be a non-empty finite set and $W\subseteq C^\omega$ be an energy condition over a totally ordered group $(G, +, \le)$, defined by a valuation of colors $\val\colon C\to G$. Then $\per(W) = \{x\in C^+\mid \val(x) < 0\}$.
\end{proposition}
\begin{proof}
 Define $P = \{x\in C^+\mid \val(x) < 0\}$. By the argument from the derivation of Theorem \ref{thm:pos}, we have $W = W_P$. Moreover, it was shown there that $P$ and $C^+\setminus P$ are closed under concatenations and cyclic shifts. Finally, by Proposition \ref{prop:technical}, we have that $P = \per(W_P) = \per(W)$.
\end{proof}

Thus, Conjecture \ref{our_conj} is equivalent to the following conjecture.
\begin{conjecture}
\label{our_conj2}
Let $C$ be any non-empty finite set. Then for any $P\subseteq C^+$ such that $P$ and $C^+\setminus P$ are closed under concatenations and cyclic shifts there exist a totally ordered group $(G, +, \le)$ and a valuation of colors $\val\colon C\to G$ such that $P = \{x\in C^+\mid \val(x) < 0\}$.
\end{conjecture}

It might be concerning that $P$ and  $C^+\setminus P$ are interchangeable in Conjecture \ref{our_conj2}, while $\val$ treats them asymmetrically. Namely, we require it to be negative on $P$ and \emph{non-negative} on $C^+\setminus P$. However, $\val$ can always be made strictly positive on $C^+\setminus P$. Namely, instead of $G$, consider the direct product $G\times\mathbb{Z}$ with the lexicographic order, and define a new valuation of colors $\val^\prime\colon C\to G\times\mathbb{Z}$,  $\val^\prime(c) = (\val(c), 1)$.

Finally, we notice that our conjecture can be reduced to a reasoning about free groups.
\begin{definition}
A subset $S$ of a group $G$ is called an \textbf{invariant sub-semigroup} of $G$ if the following two conditions hold:
\begin{enumerate}[label=\Alph*)]
\item $xy\in S$ for all $x, y\in S$ (closure under multiplications);
\item $g x g^{-1} \in S$ for all $g\in G, x\in S$ (closure under conjugations with elements of $G$).
\end{enumerate}
\end{definition}
\begin{conjecture}
\label{our_conj3}
Consider an arbitrary non-empty finite set $C$ and  any $P\subseteq C^+$ such that $P$ and $C^+\setminus P$ are closed under concatenations and cyclic shifts. Then there exists an invariant sub-semigroup $S$
of the free group $F_C$ such that, first, $C^+\setminus P$ is a subset of $S$, second, $P$ is disjoint with $S$, and third, for every $g\in F_C$ either $g\in S$ or $g^{-1}\in S$ (in particular, $S$ must have the neutral element).
\end{conjecture}

\begin{proposition}
Conjecture \ref{our_conj2} is equivalent to Conjecture \ref{our_conj3}.
\end{proposition}
\begin{proof}

Consider an arbitrary non-empty finite set $C$.
It is sufficient to show that for any $P\subseteq C^+$ the following two conditions are equivalent:
\begin{enumerate}[label=\Alph*)]
\item  there exist a totally ordered group $(G, +, \le)$ and a valuation of colors $\val\colon C\to G$ such that $P = \{x\in C^+\mid \val(x) < 0\}$.

\item there exists an invariant sub-semigroup $S$
of the free group $F_C$ such that, first, $C^+\setminus P$ is a subset of $S$, second, $P$ is disjoint with $S$, and third, for every $g\in F_C$ either $g\in S$ or $g^{-1}\in S$.
\end{enumerate}
We first establish $\mbox{A)}\implies\mbox{B)}$. Extend $\val$ to a homomorphism from $F_C$ to $G$ by setting $\val(c^{-1}) = -\val(c)$ for $c\in C$.
Set $S = \{g\in F_C\mid \val(g) \ge 0\}$. It is easy to check that all conditions on $S$ are satisfied.

Now we establish $\mbox{B)}\implies\mbox{A)}$. Let $S$ be as in B). Consider a binary relation $\sim$ on $F_C$, defined by $f\sim g\iff fg^{-1}, gf^{-1}\in S$ for $f, g\in F_C$. A fact that $S$ is an invariant sub-semigroup with the neutral element implies that $\sim$ is a congruence on the group $F_C$. Let $G = F_C/\sim$ be the corresponding quotient group. Now, consider a binary relation $\preceq$ on $F_C$, defined by $f\preceq g\iff gf^{-1}\in S$ for $f, g\in F_C$ (observe that $f\sim g \iff f\preceq g, g\preceq f$). It is easy to see that $\preceq$ is correctly defined over $F_C/\sim$, whose elements are equivalence classes of $\sim$. More formally, it holds that if $a\sim b, x\sim y$, then $a\preceq x \iff b\preceq y$ (it can again be derived from the fact that $S$ is an invariant sub-semigroup). It is also routine to check that $\preceq$ defines a total ordering on $G$. We need a condition that either $g\in S$ or $g^{-1}\in S$ for every $g\in F_C$ only to show the totality of our order. Namely, to show that there are no $f, g\in F_C$ with $f\not\preceq g$ and $g\not\preceq f$, we notice that otherwise neither $g f^{-1}$ nor $f g^{-1} = (g f^{-1})^{-1}$ are in $S$. Observe that the equivalence class of $g\in F_C$ w.r.t.~$\sim$ is non-negative in $(G, \preceq)$ if and only if $g\in S$. Now, recall that $C^+\setminus P$ is a subset of $S$ and $P$ is disjoint with $P$.
Hence, if we consider a valuation of colors $\val\colon C\to G$, which maps $c\in C$ to its equivalence class w.r.t.~$\sim$, then $P$ would be the set of words from $C^+$ whose valuation is negative w.r.t.~$\preceq$.
\end{proof}

\end{document}